\documentclass[a4paper,10pt,english]{smfart}
\usepackage[english]{babel}
\usepackage[OT2,T1]{fontenc}
\DeclareSymbolFont{cyrletters}{OT2}{wncyr}{m}{n}
\DeclareMathSymbol{\Sha}{\mathalpha}{cyrletters}{"58}
\usepackage{lmodern}
\usepackage{graphicx}
\usepackage{tabularx}
\usepackage{amsmath,amsthm,amssymb,amsfonts}
\usepackage[a4paper,left=4cm,right=4cm,top=4cm,bottom=4cm]{geometry}
\numberwithin{equation}{section}
\title[Rational points on quadratic twists of a given elliptic curve]
{Height of rational points on quadratic twists of a given elliptic curve}
\author{Pierre Le Boudec}
\subjclass{$11$D$45$, $11$G$05$, $14$G$05$}
\address{EPFL SB MATHGEOM TAN \\ MA C$3$ $604$ (B\^{a}timent MA)\\ Station $8$ \\ \text{CH-$1015$} Lausanne \\ Switzerland}
\email{pierre.leboudec@epfl.ch}

\begin{document}

\makeatletter
\def\imod#1{\allowbreak\mkern10mu({\operator@font mod}\,\,#1)}
\makeatother

\newtheorem{lemma}{Lemma}
\newtheorem{theorem}{Theorem}
\newtheorem{corollary}{Corollary}
\newtheorem{proposition}{Proposition}
\newtheorem{conjecture}{Conjecture}
\newtheorem{conj}{Conjecture}
\renewcommand{\theconj}{\Alph{conj}}

\newcommand{\vol}{\operatorname{vol}}
\newcommand{\D}{\mathrm{d}}
\newcommand{\rank}{\operatorname{rank}}
\newcommand{\Pic}{\operatorname{Pic}}
\newcommand{\Gal}{\operatorname{Gal}}
\newcommand{\meas}{\operatorname{meas}}
\newcommand{\Spec}{\operatorname{Spec}}
\newcommand{\eff}{\operatorname{eff}}
\newcommand{\rad}{\operatorname{rad}}
\newcommand{\sq}{\operatorname{sq}}
\newcommand{\tors}{\operatorname{tors}}
\newcommand{\Cl}{\operatorname{Cl}}

\begin{abstract}
We formulate a conjecture about the distribution of the canonical height of the lowest non-torsion rational point on a quadratic twist of a given elliptic curve, as the twist varies. This conjecture seems to be very deep and we can only prove partial results in this direction.
\end{abstract}

\maketitle

\tableofcontents

\section{Introduction}

\subsection{Rational points on quadratic twists}

Let $E$ be the elliptic curve defined over $\mathbb{Q}$ by the Weierstrass equation
\begin{equation*}
y^2 = x^3 + A x + B,
\end{equation*}
where $(A,B) \in \mathbb{Z}^2$ satisfies $4 A^3 + 27 B^2 \neq 0$. For every squarefree integer $d \geq 1$, we denote by
$E_d$ the quadratic twist of $E$ defined over $\mathbb{Q}$ by the equation
\begin{equation}
\label{Elliptic equation}
d y^2 = x^3 + A x + B.
\end{equation}
From now on, we view $A$ and $B$ as being fixed, and $d$ as a varying parameter. In particular, the dependences on $A$ and $B$ of the constants involved in the notations $O$, $\ll$ and $\gg$ will not be specified.

The celebrated Mordell-Weil Theorem states that the abelian group $E_d(\mathbb{Q})$ is finitely generated. In other words, there exists a non-negative integer $\rank E_d(\mathbb{Q})$, the algebraic rank of the curve $E_d$ over $\mathbb{Q}$, such that
\begin{equation*}
E_d(\mathbb{Q}) \simeq E_d(\mathbb{Q})_{\tors} \times \mathbb{Z}^{\rank E_d(\mathbb{Q})},
\end{equation*}
where $E_d(\mathbb{Q})_{\tors}$ is a finite abelian group.

Let $\hat{h}_{E_d}$ be the canonical height on $E_d$. The goal of this article is to study the distribution, as $d$ varies, of the quantity $\eta_d(A,B)$ defined by
\begin{equation*}
\log \eta_d(A,B) = \min \{ \hat{h}_{E_d}(P), P \in E_d(\mathbb{Q}) \setminus E_d(\mathbb{Q})_{\tors} \},
\end{equation*}
if $\rank E_d(\mathbb{Q}) \geq 1$ and $\eta_d(A,B) = \infty$ if $\rank E_d(\mathbb{Q}) = 0$.

Let us recall the conjecture of Goldfeld (see \cite{MR564926}) about the average order of $\rank E_d(\mathbb{Q})$ as $d$ varies. Let $\mathcal{S}(X)$ be the set of positive squarefree integers up to $X$. Goldfeld's Conjecture states that
\begin{equation}
\label{Goldfeld}
\sum_{d \in \mathcal{S}(X)} \rank E_d(\mathbb{Q}) \sim \frac1{2} \# \mathcal{S}(X). 
\end{equation}
Let $L(E_d,s)$ denote the Hasse-Weil $L$-function associated to the curve $E_d$ and let
$\rank_{\textrm{an}} E_d(\mathbb{Q})$ be the order of the zero of $L(E_d,s)$ at the central point. Recall that the Parity Conjecture asserts that $\rank E_d(\mathbb{Q}) = \rank_{\textrm{an}} E_d(\mathbb{Q}) \imod{2}$. Together with the conjectural estimate \eqref{Goldfeld}, it implies that, for $\iota \in \{ 0, 1 \}$, we have
\begin{equation}
\label{rank at most 1}
\# \{d \in \mathcal{S}(X), \rank E_d(\mathbb{Q}) = \iota \} \sim \frac1{2} \# \mathcal{S}(X),
\end{equation}
and
\begin{equation}
\label{rank at least 2}
\# \{d \in \mathcal{S}(X), \rank E_d(\mathbb{Q}) \geq 2 \} = o(X).
\end{equation}
The estimates \eqref{rank at most 1} and \eqref{rank at least 2} are widely believed. In particular, they are supported by the Katz-Sarnak Philosophy (see \cite{MR1659828}) about zeros of $L$-functions and also by Random Matrix Theory heuristics (see for instance \cite{MR1956231}).

The conjectural estimate \eqref{rank at least 2} states that the proportion of curves $E_d$ whose rank is at least $2$ is negligible, and we work under the convention that $\eta_d(A,B) = \infty$ if $\rank E_d(\mathbb{Q}) = 0$. As a result, in what follows, we restrict our investigation of $\eta_d(A,B)$ to the curves $E_d$ which have rank $1$.

\subsection{Analogy between quadratic twists and number fields - A Conjecture}

It is very instructive to describe the analogy between quadratic twists of a given elliptic curve and number fields (see for instance \cite[Section $1$]{MR2322355}). According to this analogy, rank one quadratic twists correspond to real quadratic fields, and the equation \eqref{Elliptic equation} corresponds to the Pell equation.

Let $D \geq 1$ be a fundamental discriminant, and let $\Cl(D)$ and $\varepsilon_D$ respectively denote the class group and the fundamental unit of the real quadratic field $\mathbb{Q}(\sqrt{D})$. Describing precisely the distribution of $\varepsilon_D$ is considered as being extremely difficult, in particular because it is linked to the celebrated Class Number One problem for real quadratic fields. Indeed, if we let $\mathcal{D}(X)$ be the set of positive fundamental discriminants up to $X$, then it is known (see \cite{MR1210518}) that there exists a constant $C>0$ such that
\begin{equation}
\label{Pell average}
\sum_{D \in \mathcal{D}(X)} \# \Cl(D) \log \varepsilon_D \sim C X^{3/2}.
\end{equation}
Let us note that the corresponding formula for positive discriminants (not necessarily fundamental) goes back to Siegel \cite{MR0012642}. In the asymptotic formula \eqref{Pell average}, the two quantities $\# \Cl(D)$ and $\log \varepsilon_D$ are inextricably mixed and no one has ever been able to separate them. 

At the beginning of the eighties, Hooley \cite{MR765829} and Sarnak \cite{MR675187}, \cite{MR814010} have, at the same time but independently, studied this problem. Their investigations led people to believe that, most of the time, $\varepsilon_D$ should be huge compared to $D$. In particular, as recently remarked by Fouvry and Jouve (see \cite[Equation $(3)$]{MR3069056}), their conjectures would imply the following.

\begin{conj}
\label{Pell Conjecture}
Let $\varepsilon > 0$ be fixed. For almost every fundamental discriminant $D \geq 1$, we have
\begin{equation*}
\varepsilon_D > e^{D^{1/2 - \varepsilon}}.
\end{equation*}
\end{conj}

Let us note that Conjecture \ref{Pell Conjecture} and the asymptotic formula \eqref{Pell average} agree with the Cohen-Lenstra heuristics \cite{MR756082} which predict that $\# \Cl(D)$ should be small very often, and even equal to $1$ for a positive proportion of $D$'s.

Let us now explain why $\varepsilon_D$ and $\eta_d(A,B)$ should have similar distributions. We recall that we are only concerned with the curves $E_d$ whose rank is equal to $1$.

An asymptotic formula analog of \eqref{Pell average} conjecturally arises from averaging over squarefree integers $d \geq 1$ the central values $L'(E_d,1/2)$. Indeed, it is known that the average order of $L'(E_d,1/2)$ has size $\log d$ (see \cite{MR1038358}, \cite{MR1109350} and \cite{MR1081731}). In addition, recall that the full Birch and Swinnerton-Dyer Conjecture predicts that $L'(E_d,1/2)$ is essentially equal to $d^{-1/2} \# \Sha(E_d) \log \eta_d(A,B)$, where $\Sha(E_d)$ denotes the
Tate-Shafarevich group of the curve $E_d$. Therefore, it is reasonable to expect that there exists a constant $C_E>0$ such that
\begin{equation}
\label{Elliptic average}
\sum_{\substack{d \in \mathcal{S}(X) \\ \rank E_d(\mathbb{Q}) = 1}} \# \Sha(E_d) \log \eta_d(A,B) \sim C_E X^{3/2} \log X.
\end{equation}

The similarities between the asymptotic formulas \eqref{Pell average} and \eqref{Elliptic average} are remarkable. In particular, the two quantities $\# \Sha(E_d)$ and $\log \eta_d(A,B)$ also seem to be very hard to separate.

Delaunay \cite{MR1837670} has carried out the Cohen-Lenstra heuristics to determine the distribution of $\# \Sha(E_d)$ for curves $E_d$ which have rank $1$. He obtained that $\# \Sha(E_d)$ should be small very often, and even equal to $1$ for a positive proportion of $d$'s. In addition, it is to be noted that the recent work of Bhargava, Kane, Lenstra, Poonen and Rains \cite{Bhargava}, which uses different methods, leads to the same predictions.

These observations led Delaunay \cite[Conjecture $7.1$]{MR2173383} to conjecture that the average order of $\log \eta_d(A,B)$ for curves $E_d$ with rank equal to $1$ should be at least $d^{1/2 - c / \log \log d}$ for some absolute constant $c > 0$. Guided by the analogy described above and Conjecture~\ref{Pell Conjecture}, we go further in this direction and conjecture that for any fixed
$\varepsilon > 0$, almost every squarefree integer $d \geq 1$ for which $\rank E_d(\mathbb{Q}) = 1$ satisfies
\begin{equation*}
\eta_d(A,B) > e^{d^{1/2 - \varepsilon}}.
\end{equation*}

As previously explained, the proportion of curves with rank at least $2$ is conjectured to be negligible so we are led to the following analog of Conjecture \ref{Pell Conjecture}.

\begin{conjecture}
\label{Main Conjecture}
Let $(A,B) \in \mathbb{Z}^2$ be such that $4 A^3 + 27 B^2 \neq 0$, and let $\varepsilon > 0$ be fixed. For almost every squarefree integer $d \geq 1$, we have
\begin{equation*}
\eta_d(A,B) > e^{d^{1/2 - \varepsilon}}.
\end{equation*}
\end{conjecture}

Lang conjectured an upper bound for the canonical height of the lowest non-torsion rational point on an elliptic curve (see \cite[Conjecture $3$]{MR717593}), and it is implicit in his work that this upper bound should be almost optimal for most curves. It is worth noting that Conjecture~\ref{Main Conjecture} is in agreement with this general philosophy.

Conversely, Conjecture \ref{Main Conjecture} gives conjectural information about the size of $\# \Sha(E_d)$ for curves $E_d$ which have rank $1$. More precisely, if we assume the full Birch and Swinnerton-Dyer Conjecture, and also that a positive proportion of curves $E_d$ have rank $1$, and finally Conjecture \ref{Main Conjecture}, then one can show that for any fixed
$\varepsilon > 0$, almost every squarefree integer $d \geq 1$ such that $\rank E_d(\mathbb{Q}) = 1$ satisfies
\begin{equation*}
\# \Sha(E_d) < d^{\varepsilon}.
\end{equation*}

\subsection{Results towards Conjecture \ref{Pell Conjecture} and Conjecture \ref{Main Conjecture}}

Conjecture \ref{Pell Conjecture} is far out of reach. Indeed, Hooley \cite[Corollary of Theorem $1$]{MR765829} was only able to prove that for any fixed $\varepsilon > 0$, almost every discriminant (not necessarily fundamental) $D \geq 1$ satisfies
$\varepsilon_D > D^{3/2 - \varepsilon}$. Then, Fouvry and Jouve \cite[Corollary $1$]{MR3030681} improved the exponent $3/2$ to $7/4$ and recently, Reuss \cite[Corollary $6$]{Reuss} improved it to $3$. This should be compared with the trivial lower bound
$\varepsilon_D \gg D^{1/2}$.

The modesty of these results is a good clue of how deep Conjecture \ref{Pell Conjecture} must lie. The goal of this article is to establish analogs of these results for our problem.

It is easy to check that for every squarefree integer $d \geq 1$, we have $\eta_d(A,B) \gg d^{1/8}$ (see Section
\ref{Heights}). In addition, we will see that this lower bound is best possible. Note that Silverman has proved that we always have such a lower bound for twists of abelian varieties in general (see \cite[Theorem $6$]{MR747871}).

In the general case, we can prove the following result.

\begin{theorem}
\label{Theorem general}
Let $(A,B) \in \mathbb{Z}^2$ be such that $4 A^3 + 27 B^2 \neq 0$, and let $\varepsilon > 0$ be fixed. For almost every squarefree integer $d \geq 1$, we have
\begin{equation*}
\eta_d(A,B) > d^{1/4 - \varepsilon}.
\end{equation*}
\end{theorem}

The main purpose of this article is to study an example for which Theorem \ref{Theorem general} can be improved. More precisely, we consider the elliptic curve linked to the congruent number problem, that is to say the case $(A,B) = (-1,0)$. However, it is worth pointing out that our method would actually apply to any elliptic curve with full rational $2$-torsion. We obtain the following result.

\begin{theorem}
\label{Theorem congruent}
Let $\varepsilon > 0$ be fixed. For almost every squarefree integer $d \geq 1$, we have
\begin{equation*}
\eta_d(-1,0) > d^{5/8 - \varepsilon}.
\end{equation*}
\end{theorem}

To establish Theorems \ref{Theorem general} and \ref{Theorem congruent}, one is led to investigate the cardinalities
\begin{equation}
\label{Definition N}
\mathcal{N}_{\alpha}(A,B;X) = \# \{ d \in \mathcal{S}(X), \eta_d(A,B) \leq d^{1/8 + \alpha} \},
\end{equation}
and
\begin{equation}
\label{Definition N all}
\mathcal{N}_{\alpha}^{\ast}(A,B;X) = \sum_{d \in \mathcal{S}(X)}
\# \{ P \in E_d(\mathbb{Q}) \setminus E_d(\mathbb{Q})_{\tors}, \exp \hat{h}_{E_d}(P) \leq d^{1/8 + \alpha} \},
\end{equation}
where $\alpha > 0$ is fixed.

A simple observation shows that $\mathcal{N}_{\alpha}^{\ast}(A,B;X) \ll X^{1/2 + 4 \alpha}$ for any fixed $\alpha > 0$, which suffices to prove Theorem \ref{Theorem general}.

In the case $(A,B) = (-1,0)$, we use the fact that the curves $E_d$ have full rational $2$-torsion to perform complete
$2$-descents. We then use geometry of numbers methods to prove that
$\mathcal{N}_{\alpha}^{\ast}(-1,0;X) \ll X^{1/2 + \alpha + \varepsilon}$ for any fixed $\alpha > 0$ and $\varepsilon > 0$, which suffices to prove Theorem \ref{Theorem congruent}.

\subsection{Acknowledgements}

It is a great pleasure for the author to thank Peter Sarnak for his interest in this problem, and for generously sharing his thoughts and intuition. The author would also like to thank Joe Silverman for his enlightening comments on an earlier version of the manuscript.

This work was started while the author was a Postdoctoral Member of the Institute for Advanced Study, he is now a Postdoctoral Researcher at the \'{E}cole Polytechnique F\'{e}d\'{e}rale de Lausanne. The financial support and the perfect working conditions provided by these two institutions are gratefully acknowledged.

\section{Preliminaries}

\subsection{Descent arguments}

We start by proving the following result, which gives a parametrization of the rational points on the curves $E_d$ in the general case.

\begin{lemma}
\label{Descent 0}
Let $(A,B) \in \mathbb{Z}^2$ be such that $4 A^3 + 27 B^2 \neq 0$. Let also $d \geq 1$ be a squarefree integer and let
$(x,y,z) \in \mathbb{Z} \times \mathbb{Z}_{\geq 1}^2$ satisfying $\gcd(x,y,z) = 1$ and
\begin{equation*}
d y^2 z = x^3 +A x z^2 + B z^3.
\end{equation*}
Then, there is a unique way to write $x = d_1 b_1 x_1$, $z = d_1^2 b_1^3$ and $d = d_0 d_1$ where
$(d_0, d_1, b_1, x_1) \in \mathbb{Z}_{\geq 1}^3  \times \mathbb{Z}$ satisfy the conditions
$|\mu(d_0 d_1)| = 1$ and $\gcd(x_1, d_1 b_1) = 1$, and the equation
\begin{equation}
\label{Equation Descent}
d_0 y^2 = x_1^3 + A x_1 d_1^2 b_1^4 + B d_1^3 b_1^6.
\end{equation}
\end{lemma}

\begin{proof}
Let $d_1 = \gcd(d,z)$ and write $d = d_0 d_1$ and $z = d_1 z_0$ for some $(d_0, z_0) \in \mathbb{Z}_{\geq 1}^2$ satisfying
$\gcd(d_0, z_0) = 1$. We see that $d_1 \mid x^3$ and since $d_1$ is squarefree, we actually have $d_1 \mid x$. We can thus write
$x = d_1 x_0$ for some $x_0 \in \mathbb{Z}$. The equation becomes
\begin{equation*}
d_0 z_0 y^2 = d_1 \left( x_0^3 + A x_0 z_0^2 + B z_0^3 \right).
\end{equation*}
Therefore, the coprimality condition $\gcd(d_1, d_0 y) = 1$ implies $d_1 \mid z_0$, and we write $z_0 = d_1 z_1$ for some
$z_1 \in \mathbb{Z}_{\geq 1}$. We thus get
\begin{equation*}
d_0 z_1 y^2 = x_0^3 + A x_1 d_1^2 z_1^2 + B d_1^3 z_1^3 .
\end{equation*}
Let $b_1 = \gcd(x_0, z_1)$. We have $\gcd(b_1, d_0 y) = 1$ so we see that $z_1 = b_1^3$. We also write $x_0 = b_1 x_1$ for some $x_1 \in \mathbb{Z}$. We obtain the equation \eqref{Equation Descent}. Moreover, using this equation, it is easy to check that the coprimality conditions between the variables $d_0$, $d_1$, $b_1$, $x_1$ and $y$ can be summed up as $|\mu(d_0 d_1)| = 1$ and $\gcd(x_1, d_1 b_1) = 1$, which completes the proof.
\end{proof}

The following lemma describes the familiar process of complete $2$-descent in the case $(A,B) = (-1,0)$, and is the first key tool in the proof of Theorem \ref{Theorem congruent}.

\begin{lemma}
\label{Descent}
Let $d \geq 1$ be a squarefree integer and let $(x,y,z) \in \mathbb{Z}_{\neq 0} \times \mathbb{Z}_{\geq 1}^2$ satisfying
$\gcd(x,y,z) = 1$ and
\begin{equation*}
d y^2 z = x^3 - x z^2.
\end{equation*}
Then, there is a unique way to write $x = \nu d_1 d_2 b_1 b_2^2$, $y = b_2 b_3 b_4$, $z = d_1^2 b_1^3$ and $d = d_1 d_2 d_3 d_4$ where $\nu \in \{ -1, 1 \}$ and $(d_1, d_2, d_3, d_4, b_1, b_2, b_3, b_4) \in \mathbb{Z}_{\geq 1}^8$ satisfy the conditions
$|\mu(d_1 d_2 d_3 d_4)| = 1$ and $\gcd(d_1 b_1, d_2 b_2) = 1$, and the system of equations
\begin{align}
\label{System 1}
d_2 b_2^2 - \nu d_1 b_1^2 & = d_3 b_3^2, \\
\label{System 2}
\nu d_2 b_2^2 + d_1 b_1^2 & = d_4 b_4^2.
\end{align}
\end{lemma}

\begin{proof}
Using lemma \ref{Descent 0}, we get the equation
\begin{equation*}
d_0 y^2 = x_1 ( x_1 - d_1 b_1^2 ) ( x_1 + d_1 b_1^2 ).
\end{equation*}
Let us write the three factors of the right-hand side as products of a squarefree number and a square. We set
$x_1 = \nu d_2 b_2^2$, $x_1 - d_1 b_1^2 = \nu d_3 b_3^2$ and $x_1 + d_1 b_1^2 = d_4 b_4^2$ where $\nu \in \{-1, 1\}$ and
$(d_2, d_3, d_4, b_2, b_3, b_4) \in \mathbb{Z}_{\geq 1}^6$ satisfies $|\mu(d_i)| = 1$ for $i \in \{2, 3, 4 \}$. We thus get
\begin{equation*}
d_0 y^2 = d_2 d_3 d_4 b_2^2 b_3^2 b_4^2,
\end{equation*}
which implies $d_0 = d_2 d_3 d_4$ and $y = b_2 b_3 b_4$, and ends the proof.
\end{proof}

\subsection{Heights}

\label{Heights}

Let $h : \mathbb{P}^1(\overline{\mathbb{Q}}) \to \mathbb{R}_{\geq 0}$ be the logarithmic absolute Weil height and let
$h_x : \mathbb{P}^2(\overline{\mathbb{Q}}) \to \mathbb{R}_{\geq 0}$ be defined by
\begin{equation*}
h_x(x:y:z) = h(x:z)
\end{equation*}
if $(x:y:z) \neq (0:1:0)$ and $h_x(0:1:0) = 0$. It is easier for our purpose to work with the height $h_x$ so we need to find a link between the heights $\hat{h}_{E_d}$ and $h_x$. This is achieved by the following lemma.

\begin{lemma}
\label{Heights Link}
For any $P \in E_d(\mathbb{Q})$, we have
\begin{equation*}
\hat{h}_{E_d}(P) = \frac1{2} h_x(P) + O(1),
\end{equation*}
where the constant involved in the notation $O$ may depend on $E$ but neither on the point $P$ nor on the integer $d$.
\end{lemma}

\begin{proof}
Let $i : E_d(\mathbb{Q}) \to E(\mathbb{Q}(\sqrt{d}))$ be the isomorphism defined by
\begin{equation*}
i(x : y : z) = ( x : d^{1/2}y : z ),
\end{equation*}
and let $\hat{h}_E$ be the canonical height on $E$. For any $P \in E_d(\mathbb{Q})$, we have the equality
\begin{equation*}
\hat{h}_{E_d}(P) = \hat{h}_E(i(P)).
\end{equation*}
In addition, for any $Q \in E(\overline{\mathbb{Q}})$, we have
\begin{equation*}
\hat{h}_E(Q) = \frac1{2} h_x(Q) + O(1),
\end{equation*}
where the constant involved in the notation $O$ does not depend on the point $Q$. This completes the proof since we have $h_x(i(P)) = h_x(P)$ for any $P \in E_d(\mathbb{Q})$.
\end{proof}

Let $P \in E_d(\mathbb{Q}) \setminus E_d(\mathbb{Q})_{\tors}$. Replacing $P$ by $-P$ if necessary, we can assume that the point $P$ has coordinates as in lemma \ref{Descent 0}. We thus have
\begin{equation*}
h_x(P) = \log \max \{ |x_1|, d_1 b_1^2 \}.
\end{equation*}
Now, we note that the equation \eqref{Equation Descent} gives the lower bound
\begin{equation*}
\max \{ |x_1|, d_1 b_1^2 \} \gg d_0^{1/3} y^{2/3}.
\end{equation*}
As a result, we have
\begin{align*}
\max \{ |x_1|, d_1 b_1^2 \} & \gg (d_1 b_1^2)^{1/4} (d_0^{1/3} y^{2/3})^{3/4} \\
& \gg d^{1/4} b_1^{1/2} y^{1/2} \\
& \gg d^{1/4},
\end{align*}
since $b_1, y \geq 1$. Therefore, lemma \ref{Heights Link} gives the lower bound stated in the introduction
\begin{equation*}
\eta_d(A,B) \gg d^{1/8}.
\end{equation*}
In addition, this lower bound is best possible since it is attained for all squarefree integers
$d \in \{ d_1 (x_1^3 + A x_1 d_1^2 + B d_1^3), d_1, x_1 \geq 1 \}$. Note that by the work of Greaves \cite{MR1150469}, we know that there is about $X^{1/2}$ such integers up to $X$.

\subsection{Geometry of numbers}

The following lemma was recently established by the author \cite[Lemma~$4$]{Bihomogeneous} using results of Browning and
Heath-Brown based on geometry of numbers. It gives an upper bound for the number of integral solutions to a certain cubic diophantine equation, and is the second key tool in the proof of Theorem \ref{Theorem congruent}.

\begin{lemma}
\label{geometry lemma}
Let $\mathbf{f} = (f_1,f_2,f_3) \in \mathbb{Z}_{\neq 0}^3$ be a vector satisfying the conditions $\gcd(f_i, f_j) = 1$ for
$i, j \in \{1, 2, 3 \}$, $i \neq j$, and let $U_i, V_i \geq 1$ for $i \in \{ 1, 2, 3 \}$. Let also
$N_{\mathbf{f}} = N_{\mathbf{f}}(U_1,U_2,U_3,V_1,V_2,V_3)$ be the number of vectors $(u_1,u_2,u_3) \in \mathbb{Z}_{\neq 0}^3$ and
$(v_1,v_2,v_3) \in \mathbb{Z}_{\neq 0}^3$ satisfying $|u_i| \leq U_i$, $|v_i| \leq V_i$ for $i \in \{ 1, 2, 3 \}$, and the equation
\begin{equation*}
f_1 u_1 v_1^2 + f_2 u_2 v_2^2 + f_3 u_3 v_3^2 = 0,
\end{equation*}
and such that $\gcd(u_i v_i, u_j v_j) = 1$ for $i, j \in \{1, 2, 3 \}$, $i \neq j$. Let $\varepsilon > 0$ be fixed. We have the bound
\begin{equation*}
N_{\mathbf{f}} \ll_{\mathbf{f}} (U_1 U_2 U_3)^{2/3 + \varepsilon} (V_1 V_2 V_3)^{1/3}.
\end{equation*}
\end{lemma}

\section{Proofs of Theorems \ref{Theorem general} and \ref{Theorem congruent}}

\label{Proofs}

\subsection{Proof of Theorem \ref{Theorem general}}

Recall the respective definitions \eqref{Definition N} and \eqref{Definition N all} of $\mathcal{N}_{\alpha}(A,B;X)$ and
$\mathcal{N}_{\alpha}^{\ast}(A,B;X)$. Our aim is to prove that $\mathcal{N}_{\alpha}(A,B;X) = o(X)$ for fixed
$0 < \alpha < 1/8$. Since we clearly have $\mathcal{N}_{\alpha}(A,B;X) \leq \mathcal{N}_{\alpha}^{\ast}(A,B;X)$,
Theorem \ref{Theorem general} follows from the following lemma.

\begin{lemma}
\label{Lemma Theorem general}
Let $(A,B) \in \mathbb{Z}^2$ be such that $4 A^3 + 27 B^2 \neq 0$, and let $\alpha > 0$ be fixed. We have the upper bound
\begin{equation*}
\mathcal{N}_{\alpha}^{\ast}(A,B;X) \ll X^{1/2 + 4 \alpha}.
\end{equation*}
\end{lemma}

\begin{proof}
We have
\begin{equation*}
\mathcal{N}_{\alpha}^{\ast}(A,B;X) \leq \sum_{d \in \mathcal{S}(X)}
\# \{ P \in E_d(\mathbb{Q}) \setminus E_d(\mathbb{Q})_{\tors}, \exp \hat{h}_{E_d}(P) \leq X^{1/8 + \alpha} \}.
\end{equation*}
By lemma \ref{Heights Link}, we also have
\begin{equation*}
\mathcal{N}_{\alpha}^{\ast}(A,B;X) \leq \sum_{d \in \mathcal{S}(X)}
\# \{ P \in E_d(\mathbb{Q}) \setminus E_d(\mathbb{Q})_{\tors}, \exp h_x(P) \ll X^{1/4 + 2 \alpha} \}.
\end{equation*}
We note that if $(x:y:z) \in \mathbb{P}^2(\mathbb{Q})$ is a representative of
$P \in E_d(\mathbb{Q}) \setminus E_d(\mathbb{Q})_{\tors}$ then necessarily $y z \neq 0$. Lemma \ref{Descent 0} thus gives
\begin{equation*}
\mathcal{N}_{\alpha}^{\ast}(A,B;X) \leq 
2 \# \left\{ (d_0, d_1, b_1, y, x_1) \in \mathbb{Z}_{\geq 1}^4 \times \mathbb{Z},
\begin{array}{l}
|\mu(d_0 d_1)| = 1 \\
\gcd(x_1, d_1 b_1) = 1 \\
\eqref{Equation Descent} \\
d_0 d_1 \leq X \\
|x_1|, d_1 b_1^2 \ll X^{1/4 + 2 \alpha}
\end{array} \right\}.
\end{equation*}
This implies that
\begin{equation*}
\mathcal{N}_{\alpha}^{\ast}(A,B;X) \leq 2 \sum_{|x_1|, d_1 b_1^2 \ll X^{1/4 + 2 \alpha}}
\# \left\{ (d_0, y) \in \mathbb{Z}_{\geq 1}^2,
\begin{array}{l}
|\mu(d_0)| = 1 \\
\eqref{Equation Descent}
\end{array} \right\}.
\end{equation*}
For fixed $(d_1, b_1, x_1) \in \mathbb{Z}_{\geq 1}^2 \times \mathbb{Z}$, the cardinality in the right-hand side is at most $1$, so we get
\begin{equation*}
\mathcal{N}_{\alpha}^{\ast}(A,B;X) \ll  X^{1/2 + 4 \alpha},
\end{equation*}
as wished.
\end{proof}

\subsection{Proof of Theorem \ref{Theorem congruent}}

We now treat the case $(A,B) = (-1,0)$. Our aim is to prove that $\mathcal{N}_{\alpha}(-1,0;X) = o(X)$ for fixed
$0 < \alpha < 1/2$. Hence, Theorem \ref{Theorem congruent} follows from the following lemma.

\begin{lemma}
Let $\alpha > 0$ and $\varepsilon > 0$ be fixed. We have the upper bound
\begin{equation*}
\mathcal{N}_{\alpha}^{\ast}(-1,0;X) \ll X^{1/2 + \alpha + \varepsilon}.
\end{equation*}
\end{lemma}

\begin{proof}
As in the proof of lemma \ref{Lemma Theorem general}, we have
\begin{equation*}
\mathcal{N}_{\alpha}^{\ast}(-1,0;X) \leq \sum_{d \in \mathcal{S}(X)}
\# \{ P \in E_d(\mathbb{Q}) \setminus E_d(\mathbb{Q})_{\tors}, \exp h_x(P) \ll X^{1/4 + 2 \alpha} \}.
\end{equation*}
Lemma \ref{Descent} gives
\begin{equation*}
\mathcal{N}_{\alpha}^{\ast}(-1,0;X) \leq 2
\# \left\{ (\nu, \mathbf{d}, \mathbf{b}) \in \{ -1, 1 \} \times \mathbb{Z}_{\geq 1}^4 \times \mathbb{Z}_{\geq 1}^4,
\begin{array}{l}
|\mu(d_1 d_2 d_3 d_4)| = 1 \\
\gcd(d_1 b_1, d_2 b_2) = 1 \\
\eqref{System 1}, \eqref{System 2} \\
d_1 d_2 d_3 d_4 \leq X \\
d_1 b_1^2, d_2 b_2^2 \ll X^{1/4 + 2 \alpha}
\end{array} \right\},
\end{equation*}
where we have set $\mathbf{d} = (d_1, d_2, d_3, d_4)$ and $\mathbf{b} = (b_1, b_2, b_3, b_4)$.

In the following, we assume that $\nu = 1$ since the other case $\nu = -1$ can be treated similarly. For
$i \in \{ 1, 2, 3, 4 \}$, let $D_i, B_i \geq 1/2$ run over the set of powers of $2$ and let
$\mathcal{N} = \mathcal{N}(D_1, D_2, D_3, D_4, B_1, B_2, B_3, B_4)$ be the number of
$(\mathbf{d}, \mathbf{b}) \in \mathbb{Z}_{\geq 1}^4 \times \mathbb{Z}_{\geq 1}^4$ such that $D_i < d_i \leq 2 D_i$,
$B_i \leq b_i \leq 2 B_i$ for $i \in \{ 1, 2, 3, 4 \}$, and satisfying the conditions $|\mu(d_1 d_2 d_3 d_4)| = 1$,
$\gcd(d_1 b_1, d_2 b_2) = 1$, and the equations
\begin{align}
\label{System 1'}
d_2 b_2^2 - d_1 b_1^2 & = d_3 b_3^2, \\
\label{System 2'}
d_2 b_2^2 + d_1 b_1^2 & = d_4 b_4^2.
\end{align}
Note that these equations and the conditions $d_1 b_1^2, d_2 b_2^2 \ll X^{1/4 + 2 \alpha}$ imply that we also have
$d_3 b_3^2, d_4 b_4^2 \ll X^{1/4 + 2 \alpha}$. Moreover, we have
\begin{align}
\label{System 3'}
2 d_2 b_2^2 & = d_3 b_3^2 + d_4 b_4^2, \\
\label{System 4'}
2 d_1 b_1^2 & = - d_3 b_3^2 + d_4 b_4^2.
\end{align}
We have
\begin{equation*}
\mathcal{N}_{\alpha}^{\ast}(-1,0;X) \ll \sum_{\substack{D_i, B_i \\ i \in \{ 1, 2, 3, 4\}}} \mathcal{N},
\end{equation*}
where the sum is over the $D_i$, $B_i$, $i \in \{ 1, 2, 3, 4 \}$, satisfying
\begin{align}
\label{condition 1}
D_1 D_2 D_3 D_4 & \leq X, \\
\label{condition 2}
D_i B_i^2 & \ll X^{1/4 + 2 \alpha},
\end{align}
for $i \in \{ 1, 2, 3, 4 \}$.

For fixed $(d_1, d_2, b_1, b_2) \in \mathbb{Z}_{\geq 1}^4$, there is at most one $(d_4, b_4) \in \mathbb{Z}_{\geq 1}^2$ satisfying the equation \eqref{System 2'} since $d_4$ is squarefree. Note that the condition $\gcd(d_1 b_1, d_2 b_2) = 1$ and the equation \eqref{System 1'} imply that we actually have $\gcd(d_i b_i, d_j b_j) = 1$ for $i, j \in \{1, 2, 3 \}$, $i \neq j$.  Applying lemma \ref{geometry lemma} to count the number of $(d_1, d_2, d_3, b_1, b_2, b_3) \in \mathbb{Z}_{\geq 1}^6$ satisfying $D_i < d_i \leq 2 D_i$, $B_i \leq b_i \leq 2 B_i$ for $i \in \{ 1, 2, 3 \}$, $\gcd(d_i b_i, d_j b_j) = 1$ for
$i, j \in \{1, 2, 3 \}$, $i \neq j$, and the equation \eqref{System 1'}, we get
\begin{equation}
\label{Upper bound 1}
\mathcal{N} \ll X^{\varepsilon} (D_1 D_2 D_3)^{2/3} (B_1 B_2 B_3)^{1/3}.
\end{equation}
Similarly, using also the equations \eqref{System 3'} and \eqref{System 4'}, we obtain
\begin{equation}
\label{Upper bound 2}
\mathcal{N} \ll X^{\varepsilon} (D_1 D_2 D_4)^{2/3} (B_1 B_2 B_4)^{1/3},
\end{equation}
and also
\begin{equation}
\label{Upper bound 3}
\mathcal{N} \ll X^{\varepsilon} (D_1 D_3 D_4)^{2/3} (B_1 B_3 B_4)^{1/3},
\end{equation}
and finally
\begin{equation}
\label{Upper bound 4}
\mathcal{N} \ll X^{\varepsilon} (D_2 D_3 D_4)^{2/3} (B_2 B_3 B_4)^{1/3}.
\end{equation}
Note that we could have $\gcd(d_3 b_3, d_4 b_4) = 2$ but this does not change anything in the application of lemma
\ref{geometry lemma}. Combining the four upper bounds \eqref{Upper bound 1}, \eqref{Upper bound 2}, \eqref{Upper bound 3} and
\eqref{Upper bound 4}, we get
\begin{equation*}
\mathcal{N} \ll X^{\varepsilon} (D_1 D_2 D_3 D_4)^{1/2} (B_1 B_2 B_3 B_4)^{1/4}.
\end{equation*}
Summing successively over $B_i$, $i \in \{ 1, 2, 3, 4 \}$, using the condition \eqref{condition 2}, and over $D_4$ using the condition \eqref{condition 1}, we obtain
\begin{align*}
\mathcal{N}_{\alpha}^{\ast}(-1,0;X) & \ll X^{\varepsilon} \sum_{\substack{D_i, B_i \\ i \in \{ 1, 2, 3, 4\}}}
(D_1 D_2 D_3 D_4)^{1/2} (B_1 B_2 B_3 B_4)^{1/4} \\
& \ll X^{1/8 + \alpha + \varepsilon} \sum_{\substack{D_i \\ i \in \{ 1, 2, 3, 4 \}}}
(D_1 D_2 D_3 D_4)^{3/8} \\
& \ll X^{1/2 + \alpha + \varepsilon} \sum_{\substack{D_i \\ i \in \{ 1, 2, 3 \}}} 1 \\
& \ll X^{1/2 + \alpha + 2 \varepsilon},
\end{align*}
as wished.
\end{proof}

\bibliographystyle{amsalpha}
\bibliography{biblio}

\end{document}